\documentclass[a4paper,10pt]{article}
\usepackage{amsmath}
\usepackage{amssymb}
\usepackage{amsthm}
\usepackage{amscd}
\usepackage[mathscr]{eucal}

\newtheorem*{theorem}{Theorem}
\newtheorem{proposition}{Proposition}
\newtheorem{lemma}{Lemma}

\theoremstyle{definition}
\newtheorem{definition}{Definition}

\newtheorem{remark}{Remark}

\newcommand{\Gal}{\operatorname{Gal}}
\newcommand{\GL}{\operatorname{GL}}
\newcommand{\Ad}{\operatorname{Ad}}

\newcommand{\Hom}{\operatorname{Hom}}
\newcommand{\Dim}{\operatorname{dim}}
\newcommand{\Mod}{\,\mathrm{mod}\,}
\newcommand{\bk}{\mathbf{k}}

\newcommand{\Q}{\mathbb{Q}}

\title{Lifting Galois representations
over arbitrary number fields}
\begin{document}
\author{Yoshiyuki Tomiyama}
\maketitle

\begin{abstract}
It is proved that every two-dimensional residual Galois 
representation of the absolute Galois group of an arbitrary
number field lifts to a characteristic zero
$p$-adic representation, if local
lifting problems at places above $p$ are unobstructed.
\end{abstract}

\section{Introduction}
Let $\bk$ be a finite field of characteristic $p\geq 3$. Let
$K$ be a number field of finite degree over $\mathbb{Q}$
and $G_K$ its absolute Galois group $\Gal(\bar{K}/K)$. 
We consider continuous representations
\[
 \bar{\rho}:G_K\to\GL_2(\bk).
\]
The central question that we study in this paper is the existence of
a lift of $\bar{\rho}$ to $W(\bk)$, the ring of Witt vectors of $\bk$.
This question has been motivated by a conjecture 
of Serre ([S1]), that is, all odd absolutely irreducible
continuous representations $\bar{\rho}:G_{\Q}\to\GL_2(\bk)$ 
are modular of prescribed weight, level and character.
This predicts the existence of a lift to characteristic zero.
This conjecture was proved by Khare and Wintenberger in [KW1,KW2].
In [K], Khare proved the existence of lifts to $W(\bk)$ for any
$\bar{\rho}:G_K\to\GL_2(\bk)$ which are reducible.
Ramakrishna proved under very general conditions on $\bar{\rho}$ that
there exist lifts to $W(\bk)$ for $K=\Q$ in [R1,R2]. 
Gee's results ([G]) imply that there exist lifts to $W(\bk)$ for 
$p\geq 5$ and $K$ satisfying $[K(\mu_p):K]\geq 3$, 
where $\mu_p$ is the group of $p$-th roots of unity.
B\"ockle and Khare have proved the general $n$-dimensional
case for function field in [BK].
In this paper,
we extend Theorem 1 of [R1] to arbitrary number fields.
In particular, we will omit the condition $[K(\mu_p):K]\geq 3$.
Hence we can take the field $K$ to be $\mathbb{Q}(\mu_p)^+$,
the totally real subfield of $\mathbb{Q}(\mu_p)$.

For a place $v$ of $K$, let $K_v$ be
the completion of $K$ at $v$, and let
$G_v$  be its
absolute Galois group $\Gal(\bar{K}_v/K_v)$.
Let $\Ad^0\bar{\rho}$ be the set of all trace zero 
two-by-two matrices over $\bk$ with Galois action through $\bar{\rho}$
by conjugation. Our main result is the following:
\begin{theorem}
Let $K$ be a number field, and
let $\bar{\rho}:G_K\to\GL_2(\bk)$ be a continuous representation
with coefficients in a finite field $\bk$ 
of characteristic $p\geq 7$. Assume that
$H^2(G_v,\Ad^0\bar{\rho})=0$ for each places $v\mid p$. Then $\bar{\rho}$
lifts to a continuous representation $\rho:G_K\to\GL_2(W(\bk))$
which is unramified outside a finite set of places of $K$.
\end{theorem}
Our method used in the proof is essentially that of
Ramakrishna [R1,R2]. In this paper, we follow the more axiomatic
treatment presented in [T].
In Section 2, we recall a criterion of 
Ramakrishna [R2] and Taylor [T] for lifting problems.
In Section 3, we define good local lifting problems at certain
unramified places and ramified places not dividing $p$, which
will be used in Section 4. In Section 4, we prove Theorem
by using the criterion in Section 2 and local lifting problems in Section 3.
 
Throughout this paper, we assume that $p$ is a prime $\geq 7$.

\section{A criterion for lifting problems}
In this section we recall a criterion of Ramakrishna [R2] and Taylor [T] 
for a lifting from a fixed residual Galois representation
to a $p$-adic Galois representation.

Let $\bk$ be a finite field of characteristic $p$.
Throughout this paper, we consider a continuous representation
\[
 \bar{\rho}:G_K\to\GL_2(\bk).
\]
Let $S$ denote a finite set of places of $K$
containing the places above $p$, the infinite places and
the places at which $\bar{\rho}$ is ramified,
and let $K_S$ denote the maximal algebraic extension of $K$
unramified outside $S$.
Thus $\bar{\rho}$ factors through $\Gal(K_S/K)$.
Put $G_{K,S}=\Gal(K_S/K)$.
For each place $v$ of $K$, we fix an embedding $\bar{K}\subset\bar{K}_v$.
This gives a corresponding continuous homomorphism $G_v\to G_{K,S}$.

Let  $\mathscr{A}$  be the category
of complete noetherian local rings $(R,\mathfrak{m}_{R})$ 
with residue field  $\bk$
where the morphisms are homomorphisms that induce the identity map
on the residue field.

Fix a continuous homomorphism
$\delta:G_{K,S}\to W(\bk)^{\times}$, and for every 
$(R,\mathfrak{m}_{R})\in\mathscr{A}$ let $\delta_{R}$ be
the composition $\delta_{R}:G_{K,S}\to W(\bk)^{\times}\to R^{\times}$.
Suppose $\bar{\rho}:G_{K,S}\to\GL_2(\bk)$
has $\det\bar{\rho}=\delta_{\bk}$.

By a $\delta$-{\it lift} (resp.\ $\delta|_{G_v}$-{\it lift})
of $\bar{\rho}$ (resp.\ $\bar{\rho}|_{G_{v}}$) we mean
a continuous representation $\rho:G_{K,S}\to\GL_2(R)$ 
(resp.\ $\rho_v:G_v\to\GL_2(R)$) for some
$(R,\mathfrak{m}_R)\in\mathscr{A}$ such that $\rho\ (\mathrm{mod}\
\mathfrak{m}_R)=\bar{\rho}$ 
(resp.\ $\rho_v\ (\mathrm{mod}\ \mathfrak{m}_R)=\bar{\rho}|_{G_v}$) 
and det$\rho=\delta_R$
(resp.\ det$\rho_v=\delta_R|_{G_v}$).
Let $\Ad^0\bar{\rho}$ be the set of all trace zero 
two-by-two matrices over $\bk$ with Galois action through $\bar{\rho}$
by conjugation.

\begin{definition}
For a place $v$ of $K$,
we say that a pair $(\mathscr{C}_v,L_v)$, where $\mathscr{C}_v$
is a collection of $\delta|_{G_v}$-lifts of $\bar{\rho}|_{G_v}$ 
and $L_v$ is a subspace of 
$H^1(G_v,\Ad^0\bar{\rho})$, is {\it locally admissible} 
if it satisfies the following conditions:

\begin{itemize}

\item [(P1)] $(\bk,\bar{\rho}|_{G_v})\in\mathscr{C}_v$.

\item [(P2)] The set of $\delta|_{G_v}$-lifts 
in $\mathscr{C}_v$ to a fixed ring
$(R,\mathfrak{m}_R)\in\mathscr{A}$ is closed under conjugation by
elements of $1+\mathrm{M}_2(\mathfrak{m}_R)$. 

\item [(P3)] If $(R,\rho)\in\mathscr{C}_v$ and $f:R\to S$
is a morphism in $\mathscr{A}$ then $(S,f\circ\rho)\in\mathscr{C}_v$.

\item [(P4)] Suppose that $(R_1,\rho_1)$ and 
$(R_2,\rho_2)\in\mathscr{C}_v$, and $I_1$ (resp.\ $I_2$) is an
ideal of $R_1$ (resp.\ $R_2$) and that $\phi:R_1/I_1
\stackrel{\sim}{\to} R_2/I_2$
is an isomorphism such that $\phi\ (\rho_1\ (\mathrm{mod}\ I_1))
=\rho_2\ (\mathrm{mod}\ I_2)$.
Let $R_3$ be the fiber product of $R_1$ and $R_2$
over $R_1/I_1\stackrel{\sim}{\to} R_2/I_2$.
Then $(R_3,\rho_1\oplus\rho_2)\in\mathscr{C}_v$.

\item [(P5)] If $((R,\mathfrak{m}_R),\rho)$ is a $\delta|_{G_v}$-lift 
of $\bar{\rho}|_{G_v}$
such that each $(R/\mathfrak{m}_R^n,\rho\ (\mathrm{mod}\ \mathfrak{m}_R^n))
\in\mathscr{C}_v$ then 
$(R,\rho)\in\mathscr{C}_v$.

\item [(P6)] For $(R,\mathfrak{m}_R)\in\mathscr{A}$, suppose
that $I$ is an ideal of $R$ with $\mathfrak{m}_{R}I=(0)$. If
$(R/I,\rho)\in\mathscr{C}_v$ then there is a $\delta|_{G_v}$-lift 
$\tilde{\rho}$
of $\bar{\rho}|_{G_v}$ to $R$ such that
$(R,\tilde{\rho})\in\mathscr{C}_v$ and 
$\tilde{\rho}\ (\mathrm{mod}\ I)=\rho$.

\item [(P7)] Suppose that $((R,\mathfrak{m}_R),\rho_1)$ and
$(R,\rho_2)$ are $\delta|_{G_v}$-lifts of $\bar{\rho}$ with
$(R,\rho_1)\in\mathscr{C}_v$, and that $I$ is
an ideal of $R$ with $\mathfrak{m}_{R}I=(0)$ and
$\rho_1\ (\mathrm{mod}\ I)=\rho_2\ (\mathrm{mod}\ I)$. 
We shall denote by $[\rho_2-\rho_1]$
an element of $H^1(G_v,\Ad^0\bar{\rho})\otimes_{\bk}I$
defined by $\sigma\mapsto\rho_2(\sigma)\rho_1(\sigma)^{-1}-1$.
Then $[\rho_2-\rho_1]\in L_v\otimes_{\bk}I$
if and only if $(R,\rho_2)\in\mathscr{C}_v$.
\end{itemize}
\end{definition}

\begin{remark}
Note that we do regard $\mathscr{C}_v$ as a functor from
$\mathscr{A}$ to the category of sets.
\end{remark}

Let $S_{\mathrm{f}}$ be the subset of $S$ consisting of finite places.
Throughout this section, suppose that for each $v\in S_{\mathrm{f}}$
a locally admissible pair $(\mathscr{C}_v, L_v)$ is given.

Let $\bar{\chi}_p:G_K\to\bk^{\times}$ be the mod $p$ cyclotomic character.
For the $\bk[G_K]$-module $\Ad^0\bar{\rho}$, by 
$\Ad^0\bar{\rho}(i)$ for $i\in\mathbb{Z}$
we denote the twist of $\Ad^0\bar{\rho}$ 
by the $i$th tensor power of $\bar{\chi}_p$,
and by $\Ad^0\bar{\rho}^{\ast}:=\Hom(\Ad^0\bar{\rho},\bk)$
we denote its dual representation.
The $G_K$-equivariant trace pairing
$\Ad^0\bar{\rho}\times\Ad^0\bar{\rho}\to
\bk:(A,B)\mapsto\mathrm{Trace}(AB)$
is perfect. In particular,
$\Ad^0\bar{\rho}\cong\Ad^0\bar{\rho}^{\ast}$
as representations. Thus 
$\Ad^0\bar{\rho}(1)\cong\Ad^0\bar{\rho}^{\ast}(1)$
as representations. By the Tate local duality 
this induces a perfect pairing
\[
 H^1(G_v,\Ad^0\bar{\rho})\times H^1(G_v,\Ad^0\bar{\rho}(1))
\to H^2(G_v,\bk(1))\cong\bk.
\]

\begin{definition}
A $\delta$-{\it lift of type} $(\mathscr{C}_v)_{v\in S_{\mathrm{f}}}$ is a
$\delta$-lift such that $\rho|_{G_v}\in\mathscr{C}_v$
for all $v\in S_{\mathrm{f}}$.
\end{definition}

\begin{definition}
We define the {\it Selmer group} 
$H^1_{\{L_v\}}(G_{K,S},\Ad^0\bar{\rho})$ to be the kernel of the map
\[
 H^1(G_{K,S},\Ad^0\bar{\rho})\to\bigoplus_{v\in S_{\mathrm{f}}}
H^1(G_v,\Ad^0\bar{\rho})/L_v
\]
and the {\it dual Selmer group}
$H^1_{\{L_v^{\bot}\}}(G_{K,S},\Ad^0\bar{\rho}(1))$
to be the kernel of the map
\[
 H^1(G_{K,S},\Ad^0\bar{\rho}(1))\to\bigoplus_{v\in S_{\mathrm{f}}}
H^1(G_v,\Ad^0\bar{\rho}(1))/L_v^{\bot}
\]
where $L_v^{\bot}\subset H^1(G_v,\Ad^0\bar{\rho}(1))$
is the annihilator of $L_v\subset H^1(G_v,\Ad^0\bar{\rho})$ 
under the above pairing.
\end{definition}

\begin{proposition}
Keep the above notation and assumptions.
If 
\[
 H^1_{\{L_v^{\bot}\}}(G_{K,S},\Ad^0\bar{\rho}(1))=0,
\]
then there exists a $\delta$-lift of $\bar{\rho}$ to $W(\bk)$ 
of type $(\mathscr{C}_v)_{v\in S_{\mathrm{f}}}$.
\end{proposition}

\begin{proof}
By Theorem 4.50 of [H] we have the exact sequence
\begin{align*}
H^1(G_{K,S},\Ad^0\bar{\rho})&\xrightarrow{\alpha}
\bigoplus_{v\in S_{\mathrm{f}}}
H^1(G_v,\Ad^0\bar{\rho})/L_v\to
H^1_{\{L_v^{\bot}\}}(G_{K,S},\Ad^0\bar{\rho}(1))^{\ast}\\
&\to H^2(G_{K,S},\Ad^0\bar{\rho})
\xrightarrow{\beta}\bigoplus_{v\in S_{\mathrm{f}}}
H^2(G_v,\Ad^0\bar{\rho}).
\end{align*}
Consequently, we see that the map $\alpha$
is surjective and the map $\beta$
is injective.
Now we construct $\delta$-lifts $\rho_n$ of $\bar{\rho}$ to $W(\bk)/p^n$ 
of type $(\mathscr{C}_v)_{v\in S_{\mathrm{f}}}$ inductively.
By the condition (P1), there is nothing to prove for $n=1$.
Assume that there is a $\delta$-lift 
$\rho_{n-1}$ of $\bar{\rho}$ to $W(\bk)/p^{n-1}$ 
of type $(\mathscr{C}_v)_{v\in S_{\mathrm{f}}}$.
By the condition (P6), for each $v\in S_{\mathrm{f}}$ 
we can lift $\rho_{n-1}|G_v$ to a
continuous homomorphism $\rho_v:G_v\to\GL_2(W(\bk)/p^n)$
such that $(W(\bk)/p^n,\rho_v)\in\mathscr{C}_v$.
Thus we can lift $\rho_{n-1}$ to a continuous homomorphism
$\rho:G_{K,S}\to\GL_2(W(\bk)/p^n)$ by injectivity of the map $\beta$.
By surjectivity of the map $\alpha$ we may find a 
class $\phi\in H^1(G_{K,S},\Ad^0\bar{\rho})$
mapping to
\[
 ([\rho_v-\rho|_{G_v}]\Mod L_v)_{v\in S_{\mathrm{f}}}
\in\bigoplus_{v\in S_{\mathrm{f}}}
H^1(G_v,\Ad^0\bar{\rho})/L_v.
\] 
We define $\rho_n:=(1+\phi)\rho$. By the condition (P7) the
representation  $\rho_n$ is a $\delta$-lift of $\bar{\rho}$ to $W(\bk)/p^n$ 
of type $(\mathscr{C}_v)_{v\in S_{\mathrm{f}}}$.
The induction is now complete.
Then we have a $\delta$-lift of $\bar{\rho}$ to $W(\bk)$ 
of type $(\mathscr{C}_v)_{v\in S_{\mathrm{f}}}$ 
by the condition (P5) and the proposition
is proved.
\end{proof}

\section{Local lifting problems}
For a place $v$ of $K$, consider a continuous homomorphism
\[
 \bar{\rho}_v:G_v\to\GL_2(\bk).
\]
We denote by $\widehat{\varepsilon}:G_v\to W(\bk)^{\times}$ 
the Teichm\"uller lift for
any character $\varepsilon:G_v\to\bk^{\times}$
and $\widehat{\mu}\in W(\bk)$ the Teichm\"uller lift for
any element $\mu$ of $\bk$.
Let $\chi_p$ be the $p$-adic cyclotomic character.

In this section, for ramified places not dividing $p$ and
certain unramified places,
we construct a good locally admissible pairs
$(\mathscr{C}_v,L_v)$ with the 
$\delta_v:=\widehat{\det\bar{\rho}_v}\widehat{\bar{\chi}}_p^{-1}\chi_p$,
which will be used in Section 4.
Let $I_v$ be the inertia subgroup of $G_v$.
We distinguish following three cases.

\subsection{Case I}
Suppose $\bar{\rho}_v$ is unramified and $v\nmid p$.
Suppose that
\[
 \bar{\rho}_v(s)=
\left(
\begin{array}{cc}
\lambda & \lambda \\
0 & \lambda 
\end{array}
\right)
\]
and $q_v\equiv 1\mod p$,
where $\lambda$ is an element of $\bk^{\times}$ and 
$s$ is a lift of the Frobenius automorphism in $G_v/I_v$
and $q_v$ is the order of the residue field of $K_v$.
Note that any $\delta_v$-lift of $\bar{\rho}_v$ factors
through the Galois group $\Gal(K_v^{\mathrm{t}}/K_v)$ 
of the maximal tamely ramified extension
$K_v^{\mathrm{t}}$ of $K_v$.
Let $P_v$ be the wild inertia subgroup of $G_v$.
Let $t$ be a topological generator of $I_v/P_v$.
The Galois group $\Gal(K_v^{\mathrm{t}}/K_v)$ 
is generated topologically by $s$ and $t$ 
with the relation $sts^{-1}=t^{q_v}$.
We now define a homomorphism
$\rho_v:G_v\to\Gal(K_v^{\mathrm{t}}/K_v)\to\GL_2(W(\bk)[[X]])$ by
\[
 s\mapsto
\left(
\begin{array}{cc}
\widehat{\lambda}q_v & \widehat{\lambda} \\
0 & \widehat{\lambda} 
\end{array}
\right)
\]
and
\[
 t\mapsto
\left(
\begin{array}{cc}
1 & X \\
0 & 1 
\end{array}
\right).
\]
The images of $s$ and $t$ satisfy the relation $sts^{-1}=t^{q_v}$.
We define a pair $(\mathscr{C}_v,L_v)$.
The functor $\mathscr{C}_v:\mathscr{A}\to\mathbf{Sets}$
is given by 
\begin{align*}
\mathscr{C}_v(R):=\{\rho:G_v&\to\GL_2(R)\mid 
\text{there are } \alpha\in
\Hom_{\mathscr{A}}(W(\bk)[[X]],R) \text{ and } \\
&M\in 1+\mathrm{M}_2(\mathfrak{m}_R) 
\text{ such that } \rho=M(\alpha\circ\rho_v)M^{-1}\}.
\end{align*}
Moreover, if $\rho_0:G_v\to\GL_2(\bk[X]/(X^2))$ denotes the trivial lift
of $\bar{\rho}_v$, we define
a subspace $L_v\subset H^1(G_v,\Ad^0\bar{\rho}_v)$
to be the set
\[
 \{[c]\in H^1(G_v,\Ad^0\bar{\rho}_v)\mid(1+Xc)\rho_0\in
\mathscr{C}_v(\bk[X]/(X^2))\}.
\]
\begin{lemma}
We have \\
$(\mathrm{i})$ $\Dim_{\bk}L_v=\Dim_{\bk}H^1(G_v/I_v,\Ad^0\bar{\rho}_v)
=1$.\\
$(\mathrm{ii})$ The pair $(\mathscr{C}_v,L_v)$ satisfies the conditions
$\mathrm{(P1)}$-$\mathrm{(P7)}$ of Definition 1.
\end{lemma}
\begin{proof}
(i) First we prove that $\Dim_{\bk}H^1(G_v/I_v,\Ad^0\bar{\rho}_v)=1$.
By Proposition 18 of [S2] the dimension of $H^1(G_v/I_v,\Ad^0\bar{\rho}_v)$
is the same as that of $H^0(G_v,\Ad^0\bar{\rho}_v)$.
Thus it suffices to show that $H^0(G_v,\Ad^0\bar{\rho}_v)$ is
one-dimensional. This follows from
\[
 \left(
\begin{array}{cc}
\lambda & \lambda \\
0 & \lambda 
\end{array}
\right)
\left(
\begin{array}{cc}
a & b \\
c & -a 
\end{array}
\right)
\left(
\begin{array}{cc}
1/\lambda & -1/\lambda \\
0 & 1/\lambda 
\end{array}
\right)
=
\left(
\begin{array}{cc}
a+c & -2a+b-c \\
c & -(a+c) 
\end{array}
\right),
\]
where $a,b,c\in\bk$.

Next we prove that $\Dim_{\bk}L_v=1$.
Let $f_1:W[[X]]\to\bk[X]/(X^2)$ be the morphism in $\mathscr{A}$
determined by $f_1(X)=X$.
We define $\rho_1:G_v\to\GL_2(\bk[X]/(X^2))$ by the
composition $f_1\circ\rho_v$.
The images of $s$ and $t$ satisfy the relation $sts^{-1}=t^{q_v}$.
Let $c_1$ be the 1-cocycle corresponding to $\rho_1$.
The space $L_v$ is spanned by the class of $c_1$. 
Thus we have $\Dim_{\bk}L_v=1$.

(ii) The conditions (P1), (P2), (P3), (P6) and (P7) follow 
from the definition of $(\mathscr{C}_v,L_v)$.

First we prove the condition (P4). Suppose that we have rings 
$(R_1,\mathfrak{m}_{R_1}),(R_2,\mathfrak{m}_{R_2})
\in\mathscr{A}$, lifts $\rho_i\in\mathscr{C}_v(R_i)$,
ideals $I_i\subset R_i$, and an identification 
$\phi:R_1/I_1\stackrel{\sim}{\to}R_2/I_2$ under which 
$\rho_1\ (\mathrm{mod}\ I_1)=\rho_2\ (\mathrm{mod}\ I_2)$.
Take $\alpha_i\in\Hom_{\mathscr{A}}(W(\bk)[[X]],R_i)$ and
$M_i\in 1+\mathrm{M}_2(\mathfrak{m}_{R_i})$ such that 
$\rho_i=M_i(\alpha_i\circ\rho_v)M_i^{-1}$, $i=1,2$.
We claim that there exist $\alpha\in\Hom_{\mathscr{A}}(W(\bk)[[X]],R_3)$
and $M\in 1+\mathrm{M}_2(\mathfrak{m}_{R_3})$ such that 
$M(\alpha\circ\rho_v)M^{-1}=\rho_1\oplus\rho_2$.
By conjugating $\rho_1$ by some lift of $M_2\ (\mathrm{mod}\ I_2)$ to
$R_1$, we may assume that $M_2=1$.
Since $\alpha_1\circ\rho_v(s)=\alpha_2\circ\rho_v(s)$,
the matrix $M_1\ (\mathrm{mod}\ I_1)$ commutes with 
$(\alpha_1\ (\mathrm{mod}\ I_1))\circ\rho_v(s)$.
Let 
$\left(
\begin{array}{cc}
1+m_1 & m_2 \\
0 & 1+m_3 
\end{array}
\right)
\in 1+\mathrm{M}_2(\mathfrak{m}_{R_1})$
be a lift of $M_1\ (\mathrm{mod}\ I_1)$.
Put $M_1':=
\left(
\begin{array}{cc}
1+m_1 & m_2 \\
0 & 1+m_3-x 
\end{array}
\right)$,
where $x:=(q_v-1)m_2-m_1+m_3$.
Note that $x\in I_1$.
Then $M_1'\in 1+\mathrm{M}_2(\mathfrak{m}_{R_1})$
commutes with $\alpha_1\circ\rho_v(s)$.
We now replace $M_1$ by $\widetilde{M}_1:=M_1M_1'^{-1}$ and $\alpha_1$
by some $\widetilde{\alpha}_1:W(\bk)[[X]]\to R_1$ such that
$\widetilde{M}_1(\widetilde{\alpha}_1\circ\rho_v)
\widetilde{M}_1^{-1}=M_1(\alpha_1\circ\rho_v)M_1^{-1}$.
Defining $M:=(\widetilde{M}_1,1)\in 1+\mathrm{M}_2(\mathfrak{m}_{R_3})$
and $\alpha:=(\widetilde{\alpha}_1,\alpha_2):W(\bk)[[X]]\to R_3$,
the condition (P4) is verified.

Next we prove the condition (P5).
Suppose that we have a ring $R\in\mathscr{A}$ and a 
$\delta_v$-lift $\rho$ of $\bar{\rho}_v$ to $R$ such that each
$\rho\ (\mathrm{mod}\ \mathfrak{m}_{R}^n)
\in\mathscr{C}_v(R/\mathfrak{m}_{R}^n)$.
Put $\rho_n:=\rho\ (\mathrm{mod}\ \mathfrak{m}_{R}^n)$.
Take $\alpha_n\in\Hom_{\mathscr{A}}(W(\bk)[[X]],R/\mathfrak{m}_{R}^n)$ 
and $M_n\in 1+\mathrm{M}_2(\mathfrak{m}_R/\mathfrak{m}_{R}^n)$ 
such that $\rho_n=M_n(\alpha_n\circ\rho_v)M_n^{-1}$.
We claim that there exist $\alpha\in\Hom_{\mathscr{A}}(R_v,R)$
and $M\in 1+\mathrm{M}_2(\mathfrak{m}_{R})$ such that 
$M(\alpha\circ\rho_v)M^{-1}=\rho$.
Put $S_n:=\{(\alpha'_n,M'_n)\mid \rho_n=M'_n
(\alpha'_n\circ\rho_v)M'^{-1}_n\}$.
Since $\mathscr{C}_v(R/\mathfrak{m}_{R}^n)$ is finite,
$S_n$ is finite. For each $n$, $S_n$ is not empty set.
Thus $\underset{n}{\varprojlim}S_n$ is not empty set,
the condition (P5) is verified.
\end{proof}

\subsection{Case II}

Suppose $\bar{\rho}_v$ is ramified and $v\nmid p$.
In addition, suppose $\bar{\rho}_v(I_v)$ is of order prime to $p$.
Define the functor $\mathscr{C}_v:\mathscr{A}\to\mathbf{Sets}$ by 
\[
\mathscr{C}_v(R):=\{\rho:G_v\to\GL_2(R)\mid 
\rho\ (\mathrm{mod}\ \mathfrak{m}_R)=\bar{\rho}_v,
\rho(I_v)\stackrel{\sim}{\to}\bar{\rho}_v(I_v), \det\rho=\delta_v\}.
\]
Moreover, if $\rho_0:G_v\to\GL_2(\bk[X]/(X^2))$ denotes the trivial lift
of $\bar{\rho}_v$, we define 
a subspace $L_v\subset H^1(G_v,\Ad^0\bar{\rho}_v)$
to be the set
\[
 \{[c]\in H^1(G_v,\Ad^0\bar{\rho}_v)\mid(1+Xc)\rho_0\in
\mathscr{C}_v(\bk[X]/(X^2))\}.
\]

\begin{lemma}
We have \\
$(\mathrm{i})$ $\Dim_{\bk}L_v=\Dim_{\bk}H^0(G_v,\Ad^0\bar{\rho}_v)$.\\
$(\mathrm{ii})$ The pair $(\mathscr{C}_v,L_v)$ satisfies the conditions
$\mathrm{(P1)}$-$\mathrm{(P7)}$ of Definition 1.
\end{lemma}
\begin{proof}
This lemma follows from the definitions and the Schur-Zassenhaus theorem.
\end{proof}

\subsection{Case III}

Suppose $\bar{\rho}_v$ is ramified and $v\nmid p$.
In addition, suppose the order of $\bar{\rho}_v(I_v)$ is
divisible by $p$. 
By Lemma 3.1 of [G], since $p\geq 7$, 
we may assume that $\bar{\rho}_v$ is given by the form
\[
\bar{\rho}_v=
 \left(
\begin{array}{cc}
\varphi\bar{\chi}_p & \gamma \\
0 & \varphi
\end{array}
\right),
\]
for a character $\varphi:G_v\to\bk^{\times}$ and a nonzero
continuous function $\gamma:G_v\to\bk$.
The functor $\mathscr{C}_v:\mathscr{A}\to\mathbf{Sets}$
is given by 
\begin{align*}
\mathscr{C}_v(R):=\{\rho:G_v&\to\GL_2(R)\mid 
\text{there are } \widetilde{\gamma}\in
\mathrm{Map}(G_v,R) \text{ and }
M\in 1+\mathrm{M}_2(\mathfrak{m}_R) \\
&\text{ such that } \rho=M
\left(
\begin{array}{cc}
\widehat{\varphi}\chi_p & \widetilde{\gamma} \\
0 & \widehat{\varphi}
\end{array}
\right)
M^{-1}, \widetilde{\gamma}\ \mathrm{mod}\ \mathfrak{m}_R=\gamma\}.
\end{align*}
Moreover, if $\rho_0:G_v\to\GL_2(\bk[X]/(X^2))$ denotes the trivial lift
of $\bar{\rho}_v$, we define a 
subspace $L_v\subset H^1(G_v,\Ad^0\bar{\rho}_v)$
to be the set
\[
 \{[c]\in H^1(G_v,\Ad^0\bar{\rho}_v)\mid(1+Xc)\rho_0\in
\mathscr{C}_v(\bk[X]/(X^2))\}.
\]
\begin{lemma}
We have \\
$(\mathrm{i})$ $\Dim_{\bk}L_v=\Dim_{\bk}H^0(G_v,\Ad^0\bar{\rho}_v)$.\\
$(\mathrm{ii})$ The pair $(\mathscr{C}_v,L_v)$ satisfies the conditions
$\mathrm{(P1)}$-$\mathrm{(P7)}$ of Definition 1.
\end{lemma}
\begin{proof}
The proof of this lemma is almost identical argument as 
in [T, Section 1(E3)].
\end{proof}

\section{Lifting theorem over arbitrary number fields}
In this section, we give a generalization of Theorem 1
of [R1] to arbitrary number fields.

We define $\delta:G_{K,S}\to W(\bk)^{\times}$ by 
$\widehat{\det\bar{\rho}}\widehat{\bar{\chi}}_p^{-1}\chi_p$. 
Throughout this section, we consider
lifts of a fixed determinant $\delta$ and we always assume the following:
\begin{itemize}

\item The order of the image of $\bar{\rho}$ is divisible by $p$.

\end{itemize}
By the Schur-Zassenhaus theorem,
if the order of the image of $\bar{\rho}$ is prime to $p$,
we can find a lift to $W(\bk)$ of $\bar{\rho}$.
Since $p\geq 7$ and the order of the image of $\bar{\rho}$ is
divisible by $p$, we see from Section 260 of [D] that the
image of $\bar{\rho}$ is contained in the Borel subgroup of
$\GL_2(\bk)$ or the projective image of $\bar{\rho}$ is conjugate to either
$\mathrm{PGL}_2(\mathbb{F}_{p^r})$ or $\mathrm{PSL}_2(\mathbb{F}_{p^r})$
for some $r\in\mathbb{Z}_{> 0}$. In the Borel case,
by Theorem 2 of [K] we have a lift of $\bar{\rho}$ to $W(\bk)$.
Thus we may assume that the projective image of $\bar{\rho}$ is equal to
$\mathrm{PSL}_2(\mathbb{F}_{p^r})$ or $\mathrm{PGL}_2(\mathbb{F}_{p^r})$.
Then, by Lemma 17 of [R1], $\Ad^0\bar{\rho}$ is 
an irreducible $G_{K,S}$-module. (Note that one may replace
the assumption that the image of $\bar{\rho}$ contains 
$\mathrm{SL}_2(\bk)$ in [R1] with the assumption that
the projective image of $\bar{\rho}$ contains
$\mathrm{PSL}_2(\mathbb{F}_{p})$ without affecting the proof.)
The irreducibility of $\Ad^0\bar{\rho}$ implies 
that of $\Ad^0\bar{\rho}(1)$.

Let $K(\Ad^0\bar{\rho})$ be the fixed field of 
$\mathrm{Ker}(\Ad^0\bar{\rho})$. Put $E=K(\Ad^0\bar{\rho})K(\mu_p)$
and $D=K(\Ad^0\bar{\rho})\cap K(\mu_p)$.
\begin{lemma}
We have
\[
H^1(\Gal(E/K),\Ad^0\bar{\rho})=H^1(\Gal(E/K),
\Ad^0\bar{\rho}(1))=0.
\]
\end{lemma}
\begin{proof}
First we prove that $H^1(\Gal(E/K),\Ad^0\bar{\rho})=0$.
It suffices to show that
$H^1(\mathrm{SL}_2(\mathbb{F}_{p^r}),\Ad^0\bar{\rho})=0$
and $H^1(\mathrm{GL}_2(\mathbb{F}_{p^r}),\Ad^0\bar{\rho})=0$,
where $\mathrm{GL}_2(\mathbb{F}_{p^r})$ and 
$\mathrm{SL}_2(\mathbb{F}_{p^r})$ act on $\Ad^0\bar{\rho}$ by conjugation.
By Lemma 2.48 of [DDT], we see
$H^1(\mathrm{SL}_2(\mathbb{F}_{p^r}),\Ad^0\bar{\rho})=0$.
Since the index of $\mathrm{SL}_2(\mathbb{F}_{p^r})$ in
$\mathrm{GL}_2(\mathbb{F}_{p^r})$ is prime to $p$, we have
$H^1(\mathrm{GL}_2(\mathbb{F}_{p^r}),\Ad^0\bar{\rho})=0$.

Next we prove that $H^1(\Gal(E/K),\Ad^0\bar{\rho}(1))=0$.
As $D\subset K(\mu_p)$, we see $\Gal(K(\Ad^0\bar{\rho})/D)$
contains the commutator subgroup of $\Gal(K(\Ad^0\bar{\rho})/K)$.
Since the projective image of $\bar{\rho}$ is equal to
$\mathrm{PSL}_2(\mathbb{F}_{p^r})$ or
 $\mathrm{PGL}_2(\mathbb{F}_{p^r})$,
we see this commutator subgroup is just
 $\mathrm{PSL}_2(\mathbb{F}_{p^r})$.
Thus $\Gal(K(\Ad^0\bar{\rho})/K)/\mathrm{PSL}_2(\mathbb{F}_{p^r})\to
\Gal(D/K)$ is surjective, and so $[D:K]=1$ or $2$. 
Assume that  $[K(\mu_p):K]=1$, then $H^1(\Gal(E/K),\Ad^0\bar{\rho}(1))$
is isomorphic to $H^1(\Gal(E/K),\Ad^0\bar{\rho})$. Consequently
$H^1(\Gal(E/K),\Ad^0\bar{\rho}(1))=0$.

Assume that $[K(\mu_p):K]\geq 3$, or $[K(\mu_p):K]=2$ and $[D:K]=1$.
We apply the inflation-restriction sequence to $\Gal(E/K)$ and its
 normal subgroup $\Gal(E/K(\Ad^0\bar{\rho}))$.
Since $\Gal(K_S/E)$ fixes $\Ad^0\bar{\rho}(1)$
we see $\Ad^0\bar{\rho}(1)^{\Gal(E/K(\Ad^0\bar{\rho}))}=
\Ad^0\bar{\rho}(1)^{\Gal(K_S/K(\Ad^0\bar{\rho}))}$.
We get the exact sequence
\begin{align*}
0&\to H^1(\Gal(K(\Ad^0\bar{\rho})/K),
\Ad^0\bar{\rho}(1)^{\Gal(K_S/K(\Ad^0\bar{\rho}))})\to
H^1(\Gal(E/K),\Ad^0\bar{\rho}(1))\\
&\to H^1(\Gal(E/K(\Ad^0\bar{\rho})),
\Ad^0\bar{\rho}(1))^{\Gal(K(\Ad^0\bar{\rho})/K)}.
\end{align*}
The last term is trivial as $\Gal(E/K(\Ad^0\bar{\rho}))$ has
order prime to $p$.
As $\Gal(K_S/K(\Ad^0\bar{\rho}))$ acts trivially on 
$\Ad^0\bar{\rho}$ we see the action of
 $\Gal(K_S/K(\Ad^0\bar{\rho}))$ is 
$\chi_p|_{\Gal(K_S/K(\Ad^0\bar{\rho}))}$,
which is nontrivial, so 
$\Ad^0\bar{\rho}(1)^{\Gal(K_S/K(\Ad^0\bar{\rho}))}=0$.
Thus the left term in the sequence is trivial, so
$H^1(\Gal(E/K),\Ad^0\bar{\rho}(1))=0$.

Assume that $[K(\mu_p):K]=2$ and $[D:K]=2$, then we have $K(\mu_p)=D$. 
Note that $\mathrm{PSL}_2(\mathbb{F}_{p^r})$ has no non-trivial 
abelian quotients.
If the projective image
of $\bar{\rho}$ is $\mathrm{PSL}_2(\mathbb{F}_{p^r})$
for some $r\in\mathbb{Z}_{> 0}$, then $\Gal(E/K)$ has no non-trivial
abelian quotients. 
This contradicts the assumption that $[K(\mu_p):K]=2$. 
Hence, we assume that the projective image of $\bar{\rho}$
is $\mathrm{PGL}_2(\mathbb{F}_{p^r})$ for some $r\in\mathbb{Z}_{> 0}$.
Since the index of $\mathrm{PSL}_2(\mathbb{F}_{p^r})$ in
$\mathrm{PGL}_2(\mathbb{F}_{p^r})$ is equal to the index of 
$\Gal(E/K(\mu_p))$ in $\Gal(E/K)$, $\Gal(E/K(\mu_p))$ is isomorphic
to $\mathrm{PSL}_2(\mathbb{F}_{p^r})$.
We have
\[
H^1(\Gal(E/K),\Ad^0\bar{\rho}(1))\hookrightarrow
H^1(\Gal(E/K(\mu_p)),\Ad^0\bar{\rho}(1)).
\]
Since $\Ad^0\bar{\rho}(1)$ is isomorphic to $\Ad^0\bar{\rho}$
as a $\Gal(E/K(\mu_p))$-module and the cohomology group
$H^1(\Gal(E/K(\mu_p)),\Ad^0\bar{\rho})$ is zero, the proof is complete. 
\end{proof}

\begin{lemma}
If a pair $(\mathscr{C}_v, L_v)$ which is locally admissible 
is given for each $v\in S_{\mathrm{f}}$ and each elements $\phi\in 
H^1_{\{L_v^{\bot}\}}(G_{K,S},\Ad^0\bar{\rho}(1))$
and $\psi\in 
H^1_{\{L_v\}}(G_{K,S},\Ad^0\bar{\rho})$ are not zero, then
we can find a prime $w\not\in S$ and a locally admissible pair
$(\mathscr{C}_w,L_w)$ such that\\
$(1)$ $\Dim_{\bk}H^1(G_w/I_w,\Ad^0\bar{\rho})=\Dim_{\bk}L_w=1$,\\
$(2)$ the image of $\psi$ in $H^1(G_w/I_w,\Ad^0\bar{\rho})$
is not zero,\\
$(3)$ the image of $\phi$ in
 $H^1(G_w,\Ad^0\bar{\rho}(1))/L_w^{\bot}$
is not zero.
\end{lemma}

\begin{proof}
Note that Lemma 4 implies that the 
restrictions of the cocycles $\psi$ and $\phi$
are non-zero homomorphisms 
$\phi:\Gal(K_S/E)\to\Ad^0\bar{\rho}(1)$
and $\psi:\Gal(K_S/E)\to\Ad^0\bar{\rho}$.
Let $E_{\phi}$ and $E_{\psi}$ be the fixed fields of the
respective kernels. Then,
 $\Gal(E_{\phi}/E)\to\Ad^0\bar{\rho}(1)$
and $\Gal(E_{\psi}/E)\to\Ad^0\bar{\rho}$ are injective
homomorphisms of $\mathbb{F}_p[G_{K,S}]$-modules. Since $\Ad^0\bar{\rho}$
is irreducible $G_{K,S}$-module, these morphisms are bijective,
and we see $E_{\phi}\cap E_{\psi}=E_{\psi}(=E_{\phi})$ or $E$.
If the intersection is $E$, then $\Gal(E_{\phi}E_{\psi}/E)$
is isomorphic to $\Gal(E_{\phi}/E)\times\Gal(E_{\psi}/E)$.
If the intersection is $E_{\psi}$, then $\Gal(E_{\phi}E_{\psi}/E)$
is isomorphic to $\Gal(E_{\psi}/E)$ and $\Gal(E_{\phi}/E)$.
Therefore, $\Gal(E_{\phi}E_{\psi}/E)$ may be regarded as a 
$\bk[\Gal(E/K)]$-module, moreover, natural homomorphisms
$\Gal(E_{\phi}E_{\psi}/E)\to\Ad^0\bar{\rho}(1)$
and $\Gal(E_{\phi}E_{\psi}/E)\to\Ad^0\bar{\rho}$ are surjective. 
Since $\mathrm{PSL}_2(\mathbb{F}_{p^r})$ 
has no non-trivial abelian quotients, the image of the morphism 
$\widetilde{\bar{\rho}}\times\chi_p
:G_{K,S}\to\mathrm{PGL}_2(\bk)\times\bk^{\times}$ 
contains $\mathrm{PSL}_2(\mathbb{F}_{p^r})\times 1$,
where $\widetilde{\bar{\rho}}$ is the projective image 
of $\bar{\rho}$ and $\chi_p$ 
is the mod $p$ cyclotomic character of $G_{K,S}$.
Thus there is an element $\sigma\in\Gal(E/K)$ such that
$\chi_p(\sigma)=1$ and $\bar{\rho}(\sigma)=
\left(
\begin{array}{cc}
\lambda & \lambda \\
0 & \lambda 
\end{array}
\right)$, for some element $\lambda\in\bk^{\times}$.
We denote by $\widetilde{\sigma}$ a lift to $\Gal(E_{\phi}E_{\psi}/K)$
of $\sigma$. 
Let $L$ be the subset of $\Ad^0\bar{\rho}$
whose elements have the form
$\left(
\begin{array}{cc}
\ast & \ast \\
0 & \ast
\end{array}
\right)$
and let $L'$ be the subset of 
$\Ad^0\bar{\rho}(1)$
whose elements have the form
$\left(
\begin{array}{cc}
\ast & \ast \\
0 & \ast
\end{array}
\right)$.
Since $L$ and $L'$ are two-dimensional, there exists 
$\tau\in\Gal(E_{\phi}E_{\psi}/E)$ such that
$\psi(\tau)\not\in -\psi(\widetilde{\sigma})+L$
and $\phi(\tau)\not\in -\phi(\widetilde{\sigma})+L'$.

By the \v{C}ebotarev density theorem,
we can choose a place $w\not\in S$ which is unramified in 
$E_{\phi}E_{\psi}/K$ such that $\mathrm{Frob}_w=\tau\widetilde{\sigma}$.
Take $\mathscr{C}_w$ and $L_w$ as in Case I. 
By Lemma 1 of this paper and Lemma 4.8 of [BK], it follows that 
$(w,\mathscr{C}_w,L_w)$ has the desired properties.
(Note that one may replace function
fields in [BK] with number fields without affecting the proof.)
\end{proof}

\begin{lemma}
Suppose that one is given locally admissible pairs
$(\mathscr{C}_v,L_v)_{v\in S_{\mathrm{f}}}$ such that
\[
 \sum_{v\in S_{\mathrm{f}}}\Dim_{\bk}L_v\geq\sum_{v\in S}\Dim_{\bk}
H^0(G_v,\Ad^0\bar{\rho}).
\]
Then we can find a finite set of places $T\supset S$ and
locally admissible pairs $(\mathscr{C}_v,L_v)_{v\in T\smallsetminus S}$
such that
\[
 H^1_{\{L_v^{\bot}\}}(G_{K,T},\Ad^0\bar{\rho}(1))=0.
\]
\end{lemma}

\begin{proof}
Suppose that 
$0\not=\phi\in
 H^1_{\{L_v^{\bot}\}}(G_{K,S},\Ad^0\bar{\rho}(1))$.
By the assumption of the lemma and Theorem 4.50 of [H], we see that
$\Dim_{\bk}H^1_{\{L_v\}}(G_{K,S},\Ad^0\bar{\rho})\geq
\Dim_{\bk}H^1_{\{L_v^{\bot}\}}(G_{K,S},\Ad^0\bar{\rho}(1))$.
Then we can find 
$0\not=\psi\in H^1_{\{L_v\}}(G_{K,S},\Ad^0\bar{\rho})$.
Thus we can find a place $w\not\in S$ and a locally admissible pair
$(\mathscr{C}_w,L_w)$ such that\\
(1) $\Dim_{\bk}H^1(G_w/I_w,\Ad^0\bar{\rho})=\Dim_{\bk}L_w$,\\
(2) $H^1_{\{L_v\}}(G_{K,S},\Ad^0\bar{\rho})
\to H^1(G_w/I_w,\Ad^0\bar{\rho})$
is surjective,\\
(3) the image of $\phi$ in
 $H^1(G_w,\Ad^0\bar{\rho}(1))/L_w^{\bot}$
is not zero,\\
by Lemma 5.
We have an injection
\[
 H^1_{\{L_v^{\bot}\}}(G_{K,S},\Ad^0\bar{\rho}(1))
\hookrightarrow H^1_{\{L_v^{\bot}\}\cup
\{H^1(G_w,\Ad^0\bar{\rho}(1))\}}(G_{K,S\cup\{w\}},
\Ad^0\bar{\rho}(1))
\]
and we see that its cokernel has order equal to
\[
\#\mathrm{Coker}(H^1_{\{L_v\}}(G_{K,S},\Ad^0\bar{\rho})
\to H^1(G_w/I_w,\Ad^0\bar{\rho})),
\]
by applying Theorem 4.50 of [H] to
\[
H^1_{\{L_v^{\bot}\}}(G_{K,S},\Ad^0\bar{\rho}(1))
\]
and
\[
H^1_{\{L_v^{\bot}\}\cup
\{H^1(G_w,\Ad^0\bar{\rho}(1))\}}(G_{K,S\cup\{w\}},
\Ad^0\bar{\rho}(1)).
\]
Thus
\[
 H^1_{\{L_v^{\bot}\}}(G_{K,S},\Ad^0\bar{\rho}(1))=
 H^1_{\{L_v^{\bot}\}\cup
\{H^1(G_w,\Ad^0\bar{\rho}(1))\}}(G_{K,S\cup\{w\}},
\Ad^0\bar{\rho}(1)),
\]
and we obtain an exact sequence
\begin{align*}
0&\to H^1_{\{L_v^{\bot}\}\cup
\{L_w^{\bot}\}}(G_{K,S\cup\{w\}},
\Ad^0\bar{\rho}(1))\to 
H^1_{\{L_v^{\bot}\}}(G_{K,S},\Ad^0\bar{\rho}(1))\\
&\to H^1(G_w,\Ad^0\bar{\rho}(1))/L_w^{\bot}.
\end{align*}
Hence $\phi\not\in H^1_{\{L_v^{\bot}\}\cup
\{L_w^{\bot}\}}(G_{K,S\cup\{w\}},
\Ad^0\bar{\rho}(1))\subset
H^1_{\{L_v^{\bot}\}}(G_{K,S},\Ad^0\bar{\rho}(1))$.
The lemma will follow by repeating such a computation.
\end{proof}

Let $S'$ denote the set of places of $K$ consisting of
the places above $p$, the infinite places and
the places at which $\bar{\rho}$ is ramified.

\begin{proof}[Proof of Theorem]
This follows almost at once from Proposition 1 and Lemma 6.
For each places $v$ satisfying $v\in S'_{\mathrm{f}}$ and $v\nmid p$,
take $\mathscr{C}_v$ and $L_v$ as in Case II or Case III.
For places $v\mid p$, take $\mathscr{C}_v$ and $L_v$ as
the collection of
all $\delta|_{G_v}$-lifts of $\bar{\rho}|_{G_v}$
and $H^1(G_v,\Ad^0\bar{\rho})$, respectively.
By Theorem 4.52 of [H] and the assumption of Theorem, we have 
\[
\sum_{v|p}\Dim_{\bk}L_v=\sum_{v|p}
\Dim_{\bk}H^0(G_v,\Ad^0\bar{\rho})+\sum_{v|p}
[K_v:\mathbb{Q}_p]\Dim_{\bk}\Ad^0\bar{\rho}
\]
and thus we obtain
\[
\sum_{v\in S'_{\mathrm{f}}}
\Dim_{\bk}L_v\geq\sum_{v\in S'}\Dim_{\bk}
H^0(G_v,\Ad^0\bar{\rho}).
\]
\end{proof}

\end{document}